\documentclass[a4paper,10pt]{article}
\usepackage[centertags]{amsmath}
\usepackage{amsfonts}
\usepackage{amssymb}
\usepackage{amsthm}
\usepackage{dsfont}
\usepackage{mathrsfs}
\usepackage{wasysym}
\usepackage{esint}
\usepackage[usenames]{color}
\usepackage{url}
\usepackage
{hyperref}
\theoremstyle{plain}
\newtheorem{lemma}{Lemma}[section]
\newtheorem{theorem}[lemma]{Theorem}

\newtheorem{cor}[lemma]{Corollary}

\theoremstyle{definition}

\newtheorem{remark}[lemma]{Remark}

\def\barroman#1{\sbox0{#1}\dimen0=\dimexpr\wd0+1pt\relax
  \makebox[\dimen0]{\rlap{\vrule width\dimen0 height 0.06ex depth 0.06ex}%
    \rlap{\vrule width\dimen0 height\dimexpr\ht0+0.03ex\relax 
            depth\dimexpr-\ht0+0.09ex\relax}%
    \kern.5pt#1\kern.5pt}}

\newcommand{\twopartdef}[4]
{
	\left\{
		\begin{array}{ll}
			#1 & \mbox{if  } #2 \bigskip \\
			#3 & \mbox{if  } #4
		\end{array}
	\right.
}

\newcommand{\threepartdef}[6]
{
	\left\{
		\begin{array}{lll}
			#1 & \mbox{if  } #2\bigskip \\
			#3 & \mbox{if  } #4 \bigskip \\
			#5 & \mbox{if  } #6
		\end{array}
	\right.
}

\newcommand{\pl}[2]{{\frac{\partial #1}{\partial #2}}}
\newcommand{\ppl}[3]{{\frac{\partial^2 #1}{\partial #2 \partial #3}}}

\newcommand{\de}{\partial}
\newcommand{\R}{\mathbb{R}}

\newcommand{\N}{\mathbb{N}}
\newcommand{\h}{\mathcal{H}}

\newcommand{\emb}{\hookrightarrow}
\newcommand{\cemb}{\subset \subset}
\newcommand{\In}{\subset}
\newcommand{\Om}{\Omega}

\newcommand{\dl}{{\delta}}
\newcommand{\Dl}{{\Delta}}

\newcommand{\al}{{\alpha}}

\newcommand{\ed}{{\rm d}}
\newcommand{\id}{\,\,\ed}

\newcommand{\D}{{\nabla}}

\newcommand{\ti}[1]{{\tilde{#1}}}
\newcommand{\eps}{{\varepsilon}}
\newcommand{\fr}[2]{\frac{#1}{#2}}
\newcommand{\sm}{{\setminus}}

\newcommand{\vlinesub}[1]{\vline_{_{_{_{_{_{_{_{#1}}}}}}}}}
\newcommand{\la}{\langle}
\newcommand{\ra}{\rangle}

\def\XXint#1#2#3{{\setbox0=\hbox{$#1{#2#3}{\int}$}
     \vcenter{\hbox{$#2#3$}}\kern-.5\wd0}}

\title{Compactness of minimal hypersurfaces with bounded index}
\date{\today}
\author{Ben Sharp \footnote{{\sc Department of Mathematics, South Kensington Campus, Imperial College London, LONDON, SW7 2AZ, UK, ben.g.sharp@gmail.com}}}
\begin{document}
\maketitle
\begin{abstract}
We prove a compactness result for minimal hypersurfaces with bounded index and volume, which can be thought of as an extension of the compactness theorem of Choi-Schoen \cite{CSc85} to higher dimensions.
\end{abstract}

\section{Introduction}
Compactness theorems for minimal hypersurfaces are integral to proving existence results and allow for a deeper understanding of the space of minimal hypersurfaces admitted by a closed manifold. Due to the recent work of Marques-Neves \cite{MN13}  minimal hypersurfaces exist in abundance at least in the case of positive Ricci curvature for the ambient metric. Moreover, since these minimal surfaces are smooth and have bounded volume, we have some control on their index due to results of Ejiri-Micallef \cite[Theorem 4.3]{EM08} (when $n=2$) and Cheng-Tysk \cite[Theorem 3]{CT94} ($n\geq 3$)\footnote{We remark that the index bounds of Ejiri-Micallef when $n=2$ require a more subtle  treatment than the higher dimensional case of Cheng-Tysk which follows by an adapted argument of Li-Yau on estimating the index of operators in Euclidean space \cite{LY83}}. Thus it makes sense to study their compactness theory. 

In general, smooth minimal submanifolds are analytically well controlled if they have bounded volume, and point-wise bounded second fundamental form $A$ - at which point we know that such manifolds are uniformly graphical about each of their points, and that these graphs are analytically well controlled. Therefore given any sequence of minimal submanifolds with a uniform bound on volume and second fundamental form, a smooth compactness theorem easily follows. In the case of minimal surfaces $M^2\emb N^3$ the work of Choi-Schoen \cite{CSc85} gives us something stronger: if $N$ satisfies $Ric_N \geq \al >0$ then the space of closed, embedded minimal hypersurfaces with bounded genus $\gamma$ is strongly compact in the smooth topology  - there is some smooth minimal hypersurface for which a subsequence converges locally graphically to this limit. In other words, a bound on the genus gives rise to a point-wise bound on the second fundamental form, and global control on area. 

In higher dimensions $n\geq 3$ and $M^n\emb N^{n+1}$, again with $Ric_N \geq \al >0$, control on topology of $M^n$ can never give a strong compactness theory due to counter-examples for the spherical Bernstein problem of Wu-Yi Hsiang \cite{H83}. In particular, when $N=S^4$ with the round metric, there exists a sequence of smooth embedded $\{M_k^3\}\In S^4$ such that; $\h^3(M_k) \leq \Lambda <\infty$, $M_k^3\cong S^3$ (but are not great spheres in $S^4$) and $M_k$ converge (as varifolds) to a singular $M^3\In S^4$ - it has two singularities at antipodal points of $S^4$ and $M^3$ is topologically $\fr{T^2\times [0,1]}{\sim}$ where the equivalence $\sim$ pinches the top and bottom tori $T^1\times \{1\}, T^2\times \{0\}$ to points. Therefore any generalisation of Choi-Schoen to higher dimensions would involve control on some other quantities.  

If we go back to $n=2$ we remark that a genus bound gives a bound on area due to a result of Choi-Wang \cite{CW83} (when $M$ is two-sided) and Choi-Schoen \cite{CSc85} (for general $M$): if an embedded minimal surface $M^2$ in $N^3$ has Euler characteristic $\chi(M)$, letting $\pi_1(N)$ be the fundamental group of $N$ and again $Ric_N \geq \al >0$ then 
$$\h^2(M)\leq \fr{16\pi}{\al}\left( \fr{2}{|\pi_1(N)|} - \fr12 \chi(M)\right).$$ Furthermore, a bound on genus and area gives a bound on index due to results of Ejiri-Micallef \cite[Theorem 4.3]{EM08}: 
$$index(M)\leq C(N)(\h^2(M) + \gamma -1) \leq C(N)\fr{16\pi}{\al}\left( \fr{1}{|\pi_1(N)|} - \chi(M)\right)$$
where the first inequality holds without control on the codimension or a lower bound on $Ric_N$ and $C(N)$ depends linearly on the dimension, and on the second fundamental form of some isometric embedding of $N$ into a Euclidean space. Here $index(M)$ is the Morse index - the number of negative eigenvalues associated with the Jacobi (second variation) operator for minimal hypersurfaces $M\In N$: 
\begin{equation}
Q(v,v):=\int_{M} |\D^\bot v|^2 - |A|^2|v|^2 -Ric_N(v,v) \id V_M,
\end{equation} 
where $v\in\Gamma({\rm{N}}M)$ is a section of the normal bundle and $\D^\bot$ is the normal connection.  
Thus when $n=2$ we clearly have 
\begin{equation}\label{dis}
index(M) + \h^2(M) \leq C(N)\fr{16\pi}{\al}\left( \fr{1}{|\pi_1(N)|} -  \chi(M)\right).
\end{equation}
We remark that if $M$ is two-sided then we can consider $Q$ to be defined on smooth functions over $M$ since every such $v=f\nu$ for some choice of unit normal $\nu$ and $f\in C^{\infty}(M)$. We obtain 
\begin{equation}\label{secvar}
Q(v,v)=Q(f,f):=\int_{M} |\D f|^2 - (|A|^2 +Ric_N(\nu,\nu))f^2 \id V_M.
\end{equation} 
Thus, in all dimensions, if $Ric_N\geq \al >0$ then $Q(f,f)<0$ whenever $f\equiv 1$. Therefore there are no stable (index zero) two-sided minimal hypersurfaces in such $N$. By considering a totally geodesic $\R \mathbb{P}^2\In \R \mathbb{P}^3$ we can see that there exist one-sided and stable minimal hypersurfaces in manifolds of positive Ricci curvature. 

It now seems reasonable to question whether bounded volume and index is sufficient for a compactness theorem in all dimensions?

In \cite{ScS81} Schoen-Simon prove a regularity and compactness theorem for orientable hypersurfaces with bounded volume which are stationary (minimal) and stable (index zero) in arbitrary closed $N$. As an application of such a result they extend the work of Almgren \cite{A68} and Pitts \cite{P81} to prove existence and regularity of minimal hypersurfaces in closed manifolds.    

The theorems in \cite{ScS81} are proved under the assumption that $M\in IV_n(N^{n+1})$ (the space of integrable varifolds) with $\h^{n-2}(sing(M)) = 0$\footnote{in general, due to Allard's regularity Theorem (see e.g. \cite{S83} or \cite{A72}), any stationary integral varifold $M$ is smooth in an open and dense subset, denoted $reg(M)$, and we let $sing(M) = spt(M)\sm reg(M)$ denote the singular set}. The results of Schoen-Simon follow from an in-depth local analysis of stable minimal hypersurfaces, with all estimates proven about points in $M$ under normal coordinates for $N$. The assumption on the singular set allows for the use of suitable test functions in the second variation formula \eqref{secvar} - \emph{a-priori} one can only test $Q$ with functions supported away from the singular set but this can be relaxed under the restriction $\h^{n-2}(sing(M)) = 0$. The main technical result \cite[Theorem 1]{ScS81} being that if a stable hypersurface is sufficiently Hausdorff close to a plane in normal coordinates, then the surface must decompose (on a possibly smaller region) into a multi-valued smooth graph over this plane and is therefore a smooth minimal hypersurface here. The regularity and compactness theorems \cite[Theorems 2 and 3]{ScS81} then follow by a suitably adapted dimension-reduction argument of Federer coupled with the results of Simons \cite{S68} that stable cones $C$ in $\R^{n+1}$ are hyperplanes when $n\leq 6$, can have an isolated singularity when $n=7$, and $\h^{n-7+\beta}(sing(C)) = 0$ for all $\beta >0$ when $n\geq 8$. 

Here we will use the local results of Schoen-Simon in order to prove a compactness theorem for minimal hypersurfaces with bounded volume and index when $2\leq n\leq 6$. The rough idea being to use a notion of almost minimising (or almost stable) due to Pitts \cite{P81}; if we have a sequence of say, index one smooth hypersurfaces $\{M_k\}$ with bounded volume, then we know they varifold converge to some stationary limit $M$ by Allard's compactness theorem \cite{A72}. Now, pick a point in $M$ and let $B_{\eps}(x)$ be some small ambient ball. If $B_{\eps}(x)\cap M_k$ is stable for all $k$ (up to subsequence), then $M\cap B_\eps(x)$ would be almost minimising in the sense of Pitts and we have the strong regularity results of Schoen-Simon to apply here. Moreover if  $B_{\eps}(x)\cap M_k$ is unstable for all $k$ sufficiently large, there cannot be a second ball in $M_k\sm B_\eps(x)$ which is unstable - otherwise the index of the approaching $M_k$ would eventually be two c.f. Lemma \ref{sball}.

We prove moreover that the convergence must be graphical over most of $M$ - away a finite set of points - and if the number of leaves in the convergence is eventually equal to one, the convergence is smooth and graphical everywhere, with the limit having the same index bound as the approaching hypersurfaces. Furthermore if the number of leaves in the convergence is bigger than one then the limit must be stable if it is two-sided. An easy corollary of this result is that if the ambient manifold has strictly positive Ricci curvature $Ric_N \geq \al >0$, then the space of finite volume and index minimal hypersurfaces is strongly compact in the smooth topology. 
In view of \eqref{dis} we therefore have recovered the result of Choi-Schoen \cite{CSc85} when $n=2$ and extended this appropriately for higher dimensions. 

\paragraph{Acknowledgements} 
I would like to thank Andr\'e Neves for the inspiration for this work and many useful discussions. I was supported throughout by Professor Neves' European Research Council StG agreement number P34897.

\section{Statement of results}

Central to our argument is the following result of Schoen-Simon \cite[Corollary 1]{ScS81}. Here, we need only assume that $N^{n+1}$ is a complete $C^3$ Riemannian manifold. Following \cite[p. 784]{ScS81} (with slightly different notation), let $p$ be a fixed point in $N$, $\sigma$ the injectivity radius at $p$ and $\rho_0\in (0,\sigma)$.  $B_{\rho}^N(p)$ is the open geodesic ball of radius $\rho\leq \rho_0$ and centre $p$. Finally $\mu_1$ is a constant such that 
$$\sup_{B^{n+1}_{\rho_0}(0)}\left| \pl{g_{ij}}{x^k} \right| \leq \mu_1 ,\,\,\,\,\,\,\,\,\,\sup_{B^{n+1}_{\rho_0}(0)}\left| \ppl{g_{ij}}{x^k}{x^l} \right| \leq \mu_1^2   $$
where $\{x^i\}$ are normal coordinates with respect to $B^N_{\rho_0}(p)$ and $g_{ij}$ are the metric components in these coordinates. In \cite{ScS81} they use the notation that $M=reg(M)$ and $\bar{M} = reg(M)\cup sing(M) = spt(M)$ when $M$ is viewed as an integral $n$-varifold, i.e. $M\in IV_n(N^{n+1})$. We will use the more standard notation $reg(M)$, $sing(M)$ and $spt(M)$ - see \cite{S83} for an introduction to varifolds.\begin{theorem}{\cite[Corollary 1]{ScS81}}\label{ssreg}
Suppose $M\in IV_n(N^{n+1})$ is stationary with $reg(M)$ embedded and orientable.  Let $p\in spt(M)$ with $\h^n(M\cap B^N_{\rho_0}(p))<\infty$, $\h^{n-2}(sing(M)\cap B^N_{\rho_0}(p)) = 0$, and suppose that $M$ is stable in $B^N_{\rho_0}(p)$ with respect to the area functional. Then 
$$\h^{\al}(sing(M)\cap B^N_{\fr{\rho_0}{2}}(p))  = 0,\,\,\,\,\al\geq 0,\,\,\,\,\al>n-7. $$
If $n\leq 6$ and $\rho_0^{-n}\h^n(M\cap B^N_{\rho_0}(p)) \leq \mu$ then 
$$\sup_{B^N_{\fr{\rho_0}{2}}(p)} |A|\leq \fr{C}{\rho_0}$$
for some $C=C(n,\mu,\mu_1\rho_0)<\infty.$ 
\end{theorem}

\begin{remark}\label{extss}

We point out here that in the case $n=2$ a trivial extension allows one to have the same result as above but with the assumption $\h^0(sing(M))<\infty$ - we discuss this in the appendix. Moreover we mention that a highly non-trivial extension of this work, due to Neshan Wickramasekera \cite{W14} would in particular allow one to assume only that $\h^{n-1}(sing(M))=0$ for all dimensions.

\end{remark}

Here we prove the following in the smooth setting, but we note that it will hold assuming lower regularity of $N$. Let $B^N_{\rho_0}(p)$ for $p\in reg(M)$ be as above. In normal coordinates $B^{n+1}_{\rho_0}(0)$ (we will assume that $T_x M = \R^n= \{z_{n+1} = 0\} \In \R^{n+1}$), let $C_{\rho} = B^n_{\rho}(0)\times \R$ where $B^n$ denotes a ball in $\R^n=T_x M$. We will say that $M_k\to M$ smoothly and graphically at $p\in M$ if for all sufficiently large $k$, there exists some $\rho\leq \rho_0$ and smooth functions $u_k^1,\dots,u_k^L:B^n_{\rho}\to \R$ such that $M_k\cap C_{\rho}$ is the collection of graphs of the $u_k^i$ and $u_k^i\to 0$ in $C^k$ for all $k\geq 2$. Thus we can also find $\rho>0$ such that we can consider the $u_k^i$ to be defined on $M\cap B^N_{\rho}$. We note that if the convergence is smooth and graphical away from a finite set $\mathcal{Y}$ and $M$ is connected and embedded (so that $M\sm \mathcal{Y}$ is also connected) then the number of leaves in the convergence is a constant over $M\sm \mathcal{Y}$. 
\begin{theorem}\label{main}
Let $2\leq n \leq 6$ and $N^{n+1}$ be a smooth closed Riemannian manifold. If $\{M_k^n\}\In N$ is a sequence of closed, connected and embedded minimal hypersurfaces with 
$$\h^n (M_k) \leq \Lambda <\infty \,\,\,\,\,
\text{and}\,\,\,\,\,\, index(M_k)\leq I$$
for some fixed constants $\Lambda\in \R$, $I\in \N$ independent of $k$. Then up to subsequence, there exists a closed connected and embedded minimal hypersurface $M\In N$ where $M_k \to M$ in the varifold sense with$$\h^n (M) \leq \Lambda <\infty \,\,\,\,\,
\text{and}\,\,\,\,\,\, index(M)\leq I.$$ 
Now, assuming that $M_k\neq M$ eventually, we have that the convergence is smooth and graphical for all $x\in M\sm  \mathcal{Y}$ where $\mathcal{Y}=\{y_i\}_{i=1}^K\In M$ is a finite set with $K\leq I$ and the following dichotomy holds:\begin{itemize}
\item if the number of leaves in the convergence is one then $\mathcal{Y}=\emptyset$ i.e. the convergence is smooth and graphical everywhere, moreover \begin{itemize}
\item if $M$ is two-sided and $M_k\cap M=\emptyset$ eventually then $M$ is stable 
\item if $M$ is two-sided and $M_k\cap M\neq \emptyset$ eventually then $index(M)\geq 1$
\end{itemize}
\item if the number of sheets is $\geq 2$ 
\begin{itemize}
\item if $N$ has $Ric_N >0$ then $M$ cannot be one-sided
\item if $M$ is two-sided then $M$ is stable. 
\end{itemize}
\end{itemize}

\end{theorem}

\begin{remark}

The index bound is necessary to obtain the above convergence by considering a family of Scherk surfaces, or a family of Costa-Hoffmann-Meeks surfaces with genus going to infinity in a Euclidean ball. 

We also remark again that control on the topology and volume alone is not enough to ensure such a compactness result, given the examples of Hsiang \cite{H83} mentioned in the introduction. 
\end{remark}

We recall the result of Choi-Schoen \cite{CSc85}: 
\begin{theorem}{\cite[Theorem 1]{CSc85}}\label{thCS}
Let $N$ be a compact 3-dimensional manifold with positive Ricci curvature. Then the space of compact embedded minimal surfaces of fixed topological type in $N$ is compact in the $C^k$ topology for any $k\geq 2$. Furthermore if $N$ is real analytic, then this space is a compact finite-dimensional real analytic variety. 
\end{theorem}

A corollary of our main theorem is the following strong compactness result for embedded, closed minimal hypersurfaces in closed manifolds $N^{n+1}$ with $Ric_N >0$ and $2\leq n \leq 6$. We remark that, viewing \eqref{secvar}, such manifolds do not admit two-sided  stable minimal hypersurfaces. In view of \eqref{dis} we note that this recovers Theorem \ref{thCS} when we restrict to $n=2$. 

\begin{cor}\label{CS}
Let $N^{n+1}$ be a closed Riemannian manifold with $Ric_N >0$ and $2\leq n\leq 6$. Denote by $\mathfrak{M}^n(N)$ the class of closed, smooth, and embedded minimal hypersurfaces $M\In N$. Then given any $0<\Lambda <\infty$ and $I\in \mathbb{N}$ the class
$$\mathcal{M}(\Lambda,I):=\{M\in \mathfrak{M}^n_k(N) : \text{$\h^n(M)\leq \Lambda$, $index(M)\leq I$}\}$$
is compact in the $C^k$ topology for all $k\geq 2$ with single-sheeted graphical convergence to some limit $M\in \mathcal{M}(\Lambda, I)$. 

\end{cor}
\begin{remark}\label{Choi-Schoen}
Notice that by an easy argument we have the existence of some $C=C(\Lambda,I, N)$ such that for any $M\in \mathcal{M}(\Lambda, I)$ (see also \cite[Theorem 2]{CSc85})
$$\sup_{M} |A| \leq C.$$

We also have that eventually the graphical, single-sheeted convergence implies that $M_k$ is diffeomorphic to $M$ for sufficiently large $k$. Thus there exists some $\ti{C}=\ti{C}(\Lambda, I, N,i)$ such that for any $M\in \mathcal{M}(\Lambda, I)$, letting $b_i(M)$ denote the $i^{th}$ Betti number we must have 
$$b_i(M) \leq \ti{C}.$$

\end{remark}

\begin{proof}[Proof of Corollary \ref{CS}]
The proof is essentially trivial given Theorem \ref{main} - given a sequence in this class we know that \emph{if we do not have} smooth, single sheeted graphical convergence to some limit $M$ then $M$ is two sided and it is stable, which cannot happen because $Ric_N >0$.\end{proof}

\section{Supporting results}
Whilst our main concern here is with smooth hypersurfaces, we sometimes work within the class of integrable varifolds - see \cite{S83} for an introduction. Thus we will assume that we are working with $M\in IV_n(N)$ (the space of integral $n$-varifolds in $N$) and we let 
$$reg(M) = \{ x\in M | B_{\eps}(x)\cap M \,\,\text{is an embedded, connected $C^2$ manifold}\}$$ 
with $sing(M) = spt(M)\sm reg(M)$. Therefore $\h^{n-1}(sing(M))=0$ and $reg(M)$ being orientable generalises the notion of closed orientable hypersurfaces. For us $N$ is a smooth manifold, so wherever $M$ is stationary and $C^2$, it must be smooth. 

The first variation of volume with respect to $C_c^1$ vector fields $X\in \Gamma(TN)$, corresponds to the derivative of the variation of $M$, $\psi_t(M)$ induced by $X$. Therefore $M$ is stationary if 
$$0=\pl{}{t}\vlinesub{t=0} Vol(\psi_t(M)) = \int_{M} div_{M} (X) \id \h^n =: \dl_M(X)$$ for all compactly supported $C^1$ vector fields $X$, where $X(x) = \pl{\psi_t(x)}{t}|_{_{_{t=0}}}$. Since we are considering closed $N$ we can assume that $N\emb \R^N$ is isometrically embedded and consider the first and second variation formula for $M\In \R^N$ - except we obviously restrict to vector fields $X\in \Gamma(TN)$. The stability of $M$ in some open set $U\In N$, is the assumption of positivity of the second variation, derived in \cite{S83} for $M\In \R^N$:
\begin{eqnarray*}
0&\leq& \pl{^2}{t^2}\vlinesub{t=0}(Vol(\psi_t(M)))\\
&=&\int_{M} div_{M}(Z) + (div_M(X))^2 + \sum_i |(\D_{\tau_i} X)^{\bot}|^2 - \sum_{i,j} (\D_{\tau_i} X \cdot \tau_j )(\D_{\tau_j} X \cdot \tau_i) \id\h^n \\
&=:& \dl^2_{M}(X)
\end{eqnarray*}
for all $X\in \Gamma(T\R^N)$, such that $X(x)\in T_x N$ for all $x\in N$ and $X$ is compactly supported in some open $\ti{U}\In\R^{n+1}$ where $\ti{U}\cap N = U$. Here $Z(x)=\pl{^2\psi_t(x)}{t^2}|_{_{_{t=0}}} = \D_X X$ and $\{\tau_i\}$ is some orthonormal basis of $M$ at a given point $x$. 
If $X$ is supported away from $sing(M)$, then setting $X^{\bot}$ to be the projection of $X$ to the normal bundle ${\rm{N}} M\In T N$ and $\D^{\bot}$ the normal  connection in ${\rm{N}} M$ we have  
$$\dl^2_{M}(X) = \int_{M} |\D^{\bot} X^{\bot}|^2 -  |A|^2 |X^{\bot}|^2 - Ric(X^{\bot},X^{\bot}) \id\h^n=Q(X^{\bot},X^\bot).$$ 

Thus if $M$ is two-sided in $N$, setting $f=\la X, \nu\ra \in C^1_c(reg(M))$, for some unit normal $\nu$ along $reg(M)$ then it can be checked that $$\dl^2_{M}(X) = \int_{M} |\D f|^2 - (|A|^2 + Ric_N(\nu,\nu))f^2 \id\h^n= Q(f,f).$$
In the case that $M$ is two-sided and $\h^{n-2}(sing(M)) = 0$, stability is equivalent to the positivity of $Q(f,f)$ along bounded and locally Lipschitz functions $f$ (possibly non-zero over the singularities of $spt(M)$), which follows from an easy cut-off argument. 

Furthermore, if $M$ is two-sided and $n=2$ we can relax the assumption on the singular set to $\h^0(sing(M)) < \infty$ and still work with such $f$ - this is discussed in the appendix and the distinction is important for us. We remark that the results of Schoen-Simon can be trivially extended to the case that $\h^0(sing(M))<\infty$ when $n=2$ due to this fact. Actually, in \cite{W14} a much more general regularity and compactness theory is developed for stable minimal hypersurfaces, where a weaker assumption on the singular set is imposed in all dimensions; in particular one need only assume that $\h^{n-1}(sing(M))=0$ to recover the results of Schoen-Simon.

It is worth mentioning here that if $M_k\to M$ varifold converge then we have both $\dl_{M_k}(X)\to \dl_M (X)$ and $\dl^2_{M_k}(X)\to \dl^2_M (X)$ for any $X\in C_c^1(TN)$. In particular, if there are $K$ disjoint open sets $U_1\,\dots U_K$ on which $M$ is unstable, then eventually $M_k$ is unstable on each $U_i$ for $1\leq i\leq K$ and thus has $index(M_k)\geq K$ c.f. Lemma \ref{sball}. We also note Allard's compactness theorem: for a sequence of bounded mass and stationary $\{M_k\}\In IV_n(N)$ there exists some bounded mass and stationary $M\in IV_n(N)$ such that a subsequence converges in the sense of varifolds (i.e. in the sense of Radon measures on the Grassmann bundle of $n$-planes in $TN$); moreover if $\h^n(M_k)$ is uniformly bounded then they converge in Hausdorff distance also - a consequence of the monotonicity formula. Thus if $M_k$ are connected then $M$ must be as well.

\begin{lemma}\label{sball}
Suppose $M\emb N$ is a smooth, embedded hypersurface with $index(M)=I$. Given any collection of $I+1$ open sets $\{U_i\}_{i=1}^{I+1}$, $U_i\In N$ and $U_i\cap U_j = \emptyset$ when $i\neq j$ then we must have that $M$ is stable in $U_i$ for some $1\leq i\leq I+1$.  
\end{lemma}
\begin{proof}
We proceed by contradiction; suppose that there exists some collection of $\{U_i\}$ as above but for all $1\leq i\leq I+1$ there exists $X_i\in C_c^1(U_i)$ with $\dl^2_M(X_i) <0$. Letting $X^{\bot}_i$ be the projection of $X_i$ along the normal bundle of $M$, we have a collection of $I+1$ mutually orthogonal sections (all of which are non-zero) and with $Q(X^\bot_i,X^\bot_i)<0$ for all $i$. Thus $index(M)>I$, a contradiction. 
\end{proof}

\section{Proof of Theorem \ref{main}}

\begin{proof}[Proof of Theorem \ref{main}]
Allard's compactness theorem \cite[chapter 6]{A72} tells us that there exists some $M$ such that (up to subsequence) $M_k\to M$ in the varifold sense (and thus in Hausdorff distance), with $M$ stationary, integral and connected. We can moreover choose a subsequence and assume wlog that $index(M_k) = I$ for all $k$.

Before we continue the proof we need a lemma which is an easy corollary of the work of Schoen-Simon, Theorem \ref{ssreg}.

\begin{lemma}\label{nosing}
Let $x \in M$ with $M$ as above and $\sigma^N$ be the injectivity radius of $N$. Assuming the hypotheses of Theorem \ref{main}, if there exist $R\in(0,\fr{\sigma^N}{2}),k_0\in \mathbb{N}$ such that for all $k\geq k_0$ we have 
$M_k$ is stable in $B^N_R(x)$ then $M\cap B^N_{\fr{R}{2}}(x)$ is smooth and the convergence is smooth and graphical for all $y\in B^N_{\fr{R}{2}}(x)\cap M$, again up to a subsequence. 
\end{lemma}

\begin{proof}[Proof of Lemma \ref{nosing}]
Theorem \ref{ssreg} tells us that for any $y \in M_k\cap B_{\fr{3R}{4}}^N(x)$ we must have 
$$\sup_{B^N_{\fr{R}{4}}(y)} |A_k| \leq C,$$
where $C=C(N,n,\Lambda, R)<\infty$. This follows because $N$ is compact; there is a uniform upper bound on 
$$\mu_1 := \sup_{y\in N} \max \left\{\sup_{B^{n+1}_{\fr{\sigma^N}{2}}(0)}\left| \pl{g_{ij}}{x^k} \right|, \sqrtsign{\sup_{B^{n+1}_{\fr{\sigma^N}{2}}(0)}\left| \ppl{g_{ij}}{x^k}{x^l} \right|} \right\}$$
where the above is taken over normal coordinate balls centred at $y$. Moreover $M_k\cap B^N_r(x)$ must be orientable for $r<\sigma^N$ and all $k$ since $B^N_r$ is simply connected. An easy covering argument gives that $B_{\fr{3R}{4}}(x)\cap M_k$ have uniformly bounded second fundamental form and volume and a standard compactness argument finishes the proof. 

\end{proof}

\paragraph{Claim 1: The singular set of $M$ has at most $I$ points.}\label{C1}
Suppose for a contradiction that there are at least $I+1$ points $\{x_i\}_{i=1}^{I+1}\In sing(M)$ and fix 
$$\eps_0 < \fr12 \min \{\min_{i\neq j} d_g(x_i,x_j), \sigma^N\}.$$ Lemma \ref{nosing} tells us that there must be some subsequence (not re-labelled) such that $M_k$ is unstable in $B_{\eps_0}^N(x_i)$ for all $1\leq i\leq I+1$. 
By construction we have $I+1$ disjoint open sets and Lemma \ref{sball} tells us that for each $k$, $M_k$ must be stable in one of them, a contradiction. Thus we have proved the claim.

\paragraph{Claim 2: For each $x_i\in sing(M)$ there exists some $\eps_i$ such that $M$ is stable in $B^N_{\eps_i}(x_i)\sm \{x_i\}$.} 

Again we argue by contradiction. Suppose therefore that for all $\eps>0$ there exists some $X_{\eps} \in C_c^1(B^N_{\eps}(x_i)\sm \{x_i\})$ such that $\dl^2_M(X_\eps)< 0$. Pick
$$\text{$\eps_{i_1}>0$ and $X_1 \in C_c^1(B^N_{\eps_{i_1}}(x_i)\sm \{x_i\})$ satisfying $\dl^2_M(X_1)<0$. }$$ Now pick $0<\eps_{i_2}<\eps_{i_1}$ so that $spt(X_1)\In B^N_{\eps_{i_1}}(x_i)\sm B^N_{\eps_{i_2} }(x_i)$ and of course there exists some $X_{2} \in C_c^1(B^N_{\eps_{i_2}}(x_i)\sm \{x_i\})$ with $\dl^2_M(X_2)< 0$ by assumption. The same argument produces $0<\eps_{i_3}<\eps_{i_2}$, $X_3 \in C_c^1(B^N_{\eps_{i_3}}(x_i)\sm \{x_i\})$ with $\dl^2_M(X_3)< 0$ and $spt(X_2)\In B^N_{\eps_{i_2}}(x_i)\sm B^N_{\eps_{i_3} }(x_i)$. In this way we can construct infinitely many $X_s$ with disjoint supports and for which $\dl^2_M(X_s)< 0$ for all $s$. Since $M_k\to M$ in the sense of varifolds, given any $S>I$ we can find $k$ sufficiently large for which $\dl^2_{M_k}(X_s) <0$ for all $s\leq S$. Each $X_s$ restricts to each $M_k$ to produce $S$ non-zero and mutually orthogonal vector fields contributing to the index of $M_k$ and we have our contradiction.

Therefore for all $x_i\in sing(M)$, $M\sm \{x_i\}$ is stable locally about each $x_i$ so
$$\h^{n-2}(sing(M))\twopartdef{=0}{n\geq 3}{<\infty}{n=2.}$$
By the regularity results of Schoen-Simon c.f. Theorem \ref{ssreg} and Remark \ref{extss} we end up with the desired regularity of $M$.  

We now investigate what kind of convergence we have over $M$. By Lemma \ref{nosing} (and arguing exactly as in Claim 1) we know that there is a finite set $\mathcal{Y}=\{y_i\}_{i=1}^K\In M$ such that the convergence is smooth and graphical over $x\in M\sm \mathcal{Y}$ and $K\leq I$. Now for any open $\Om \cemb M\sm \mathcal{Y}$, there exists a finite cover of $\Om$ and local graphs over each element in the cover. Letting $k$ be sufficiently large (so that the Hausdorff distance of $M_k$ to $M$ is sufficiently small) we can consider the part of $M_k$ over $\Om$ to be within a normal tubular neighbourhood of $\Om$, therefore there exist layered graphs from $\Om$ parametrising $M_k$ and defined on these components. Note that the number of leaves is eventually a constant integer since $M$ is smooth and connected (and therefore $M$ with a finite number of small discs removed is also connected).

\paragraph{Claim 3: $index(M)\leq I$.}
Assume for a contradiction that $index(M)>I$, i.e. there exist $I+1$ section $\{s_i\}_{i=1}^{I+1}\In \Gamma({\rm{N}}M)$ which are $L^2$ - orthonormal and 
$$Q(s_i,s_i) = -\lambda_i \,\,\,\,\,\,\,\,\, \lambda_i>0.$$ 
Now let $X_i$ be an arbitrary $C^1$ extension of $s_i$ to the rest of $N$. We know that eventually we must have 
$\dl^2_{M_k}(X_i) < 0$ for all $i$, therefore the sections $s_i^k = X_i^{\bot_{k}}$ ($X_i$ projected onto the normal bundle of ${\rm{N}}M_k\In TN$) must be linearly dependant for all $k$ - otherwise $index(M_k) > I$.  
 
Thus there exist $\mu_1^k,\dots,\mu_{I+1}^k \in \R$ not all zero such that $\mu_1^k s_1^k + \dots + \mu_{I+1}^k s_{I+1}^k = 0.$ 
We can take a subsequence and re-order so that $\max\{|\mu_i^k|\} = |\mu_{I+1}^k|$, then dividing through by $-\mu_{I+1}^k$ and re-labelling we can assume that  
$s_{I+1}^k = \mu_1^k s_1^k + \dots + \mu_I^k s_I^k$ with all of the $|\mu_i^k|\leq 1$. 
Now, we know that 
\begin{equation}\label{proj}
\int_{M_k} |X_i^{\top_{k}}|^2 \id \h^n \to 0,\,\,\,\,\,\,\text{as $k\to \infty$},
\end{equation}
which follows from the monotonicity formula ($X_i^{\top_{k}}$ is the projection of $X$ onto $TM_k$): For all small $\eps>0$ note that over $$M\sm (\cup_{y_i\in \mathcal{Y}}B^N_{\eps}(y_i))$$ the convergence is smooth and graphical thus 
$$\int_{M_k \sm (\cup_{y_i\in \mathcal{Y}}B^N_{\eps}(y_i)) } |X_i^{\top_{k}}|^2 \id \h^n \to 0,\,\,\,\,\,\,\text{as $k\to \infty$}.$$
 By the monotonicity formula (see e.g. \cite[Remark 17.9 (3)]{S83})
 $$\int_{M_k \cap B^N_{\eps}(y_i)}  |X_i^{\top_{k}}|^2 \id \h^n \leq C \sup|X_i|^2 \eps^n$$
 where $C=C(\h^n(M_k))$, and \eqref{proj} follows. 

Therefore we also have that 
\begin{eqnarray*}
\lim_{k\to \infty}\int_{M_k} \la s_i^k, s_j^k\ra  \id \h^n&=&\lim_{k\to \infty}\left( \int_{M_k} \la X_i ,X_j \ra  - \la X_i^{\top_{k}}, X_j^{\top_{k}} \ra \id \h^n\right)\\
& =& \int_M \la s_i ,s_j\ra \id \h^n = \dl_{ij}.
\end{eqnarray*}
Hence for $i<I+1$ 
\begin{eqnarray*}
0 &=& \lim_{k\to \infty} \int_{M_k} \la s_{I+1}, s_i^k\ra  \id \h^n \\
&=& \lim_{k\to \infty} \int_{M_k} \la\mu_1^k s_1^k + \dots + \mu_I^k s_I^k,  s_i^k\ra \id \h^n  \\
&=& \lim_{k\to \infty} \mu_i^k
\end{eqnarray*} implying that $\mu_i^k \to 0$ for all $i$. 
But then it is easy to see that this implies $\lim_{k\to \infty}\int_{M_k} |s_{I+1}^k|^2 = 0$ which is a contradiction and the claim follows.

\paragraph{Claim 4: If the number of leaves in the convergence is eventually one then the convergence is smooth and graphical over all of $M$.} Initially all we know is that this is the case over $M\sm\mathcal{Y}$, but the full statement will follow from Allard's regularity theorem. Let $\eps>0$ and suppose the convergence is not graphical over some point $x\in M$. We know we can pick $r(x,\eps)>0$ sufficiently small such that (we may denote $\|M\|(U) = \int_{U\cap M} \id \h^n$ since $M$ has multiplicity one)
$$\fr{\|M\|(B^N_{r}(x))}{r^n} \leq 1 +\eps.$$
Now by varifold convergence we know that (since $r$ is sufficiently small and $\|M\|(\de B^N_r(x)) = 0$)
$$\|M_k\|(B^N_{r}(x)) \to \|M\|(B^N_{r}(x)) \leq (1 + \eps)r^n$$thus for all sufficiently large $k$ 
$$\|M_k\|(B^N_{r}(x)) \leq (1+2\eps)r^n.$$
Now pick $r>\eta=\eta(\eps, r,n)>0$ to be specified later (again by Hausdorff convergence) we get that for all sufficiently large $k$, and any $y_{k} \in M_k\cap B^N_\eta(x)$ 
$$\|M_k\|(B^N_{r-\eta}(y_k)) \leq \|M_k\|(B^N_r(x)) \leq  (1+2\eps)r^n = (1+2\eps)(r-\eta+\eta)^n \leq (1+3\eps)(r-\eta)^n$$
where now we have picked $\eta$ sufficiently small so that the last inequality holds true. 

Since $\eps$ was arbitrary, we can apply Allard's regularity theorem and conclude (in particular) that $M_k\cap B_{\eta}(x)$ is smooth for all $k$ with uniform $L^{\infty}$ control on the second fundamental form. Thus the convergence must be graphical everywhere over $M$ and we have proved the claim. 

If $N$ has positive Ricci curvature and $M$ is one-sided we lift everything to the universal cover $\pi: \ti{N}\to N$. By Frankel's theorem \cite[Generalised Hadamard Theorem]{F66}, since $Ric_{\ti{N}} >0$, the lifts $\ti{M}_k$ and $\ti{M}$ are all connected (if not, any two components must intersect which is a contradiction). Also, $\ti{N}$ is simply connected so we must have that $\ti{M}_k$ and $\ti{M}$ are orientable (and thus two-sided) embedded minimal hypersurfaces. Furthermore, we still have smooth graphical convergence of the $\ti{M}_k$  to $\ti{M}$ away from $\ti{\mathcal{Y}} = \pi^{-1}\mathcal{Y}$, and the number of sheets is again $\geq 2$. Thus we reduce to the below: we will prove that $\ti{M}$ is a stable, two-sided minimal surface in $\ti{N}$ which contradicts the positivity of the Ricci curvature.

The rest of the proof now concerns the situation where $M$ is two-sided and the graphical convergence is a single leaf or the graphical convergence has more than one leaf. 

We now follow the ideas developed in \cite{S87}, see also \cite{CM98}. 
If the number of leaves is bigger than one, or there is a single leaf always lying on one side of $M$ then we now prove that  $index(M) = 0$, i.e. $M$ is stable. Roughly speaking, if a sequence of minimal hypersurfaces converge smoothly and graphically to some fixed minimal surface, then eventually one should see a smooth variation of the limit through minimal surfaces - i.e. there must be a solution to the Jacobi equation. This is given explicitly by suitably re-normalising the geodesic distance between the limit and the approaching surfaces as in \cite{S87}. If a single sheet converges totally on one side then this is always positive, moreover if there are many sheets, then over most of $M$ we can construct a positive solution (the signed distance between the top and bottom sheet). In either case, we end up with a signed solution to the Jacobi operator. Once we know there is a strictly positive solution then this must correspond to the lowest eigenvalue and there can be no index (by standard minimax methods). 

In the case that the convergence is single sheeted and on both sides (i.e. $M_k\cap M\neq\emptyset$ - we will assume that $M_k\neq M$ otherwise the conclusion is trivial), the procedure gives a solution to the Jacobi equation which is neither strictly positive or negative, thus the first eigenvalue is negative by standard minimax arguments.

We give the details in the case that the number of sheets is $\geq 2$, and leave the case of single sheet convergence - ``one-sided convergence" and ``two-sided convergence" - mostly to the reader since it follows easily from the below. 

\paragraph{Claim 5: If the number of sheets is $\geq 2$ and $M$ is two-sided then there exists a smooth positive solution the the Jacobi equation over $M\sm\mathcal{Y}$.}
Given any compact domain $\Om\cemb M\sm \mathcal{Y}$ we know that we can find $\dl>0$ and $k$ sufficiently large such that there is some set of functions $\{u_k^1< u_k^2<\dots< u_k^L\}\in C^{\infty}(\Om)$, $L>1$ such that 
$$M_k\cap \Om_{\dl}=\{Exp_x(\nu(x) u_k^1(x)),\dots , Exp_x(\nu(x)u_k^L(x))\}
$$where $\Om_{\dl}$ is a $\dl$-normal neighbourhood of $\Om$. Now consider the following path of smooth hypersurfaces in $\Om_\dl$, for $v_k(x,t):=tu^L_k(x)+(1-t)u^1_k(x)$ given by 
$$\Sigma_k(t):=\{Exp_x(\nu(x)v_k(x,t)):x\in \Om\}$$ and notice that any compactly supported ambient vector field $Z\in C^1_c(\Om_\dl)$, gives rise to variations of $\Sigma_k(t)$, denoted 
\begin{eqnarray*}
\Sigma_k(t,s)&=&\{\Psi^k_{t,s}(x)=\psi_s(Exp_x(\nu(x)v_k(x,t))): x\in \Om\}\\
&=&(\psi_s)_\sharp(\Sigma_k(t)). 
\end{eqnarray*}
where $\psi_s$ is a family of diffeomorphisms induced by $Z$ i.e. $Z(x)=\pl{\psi_s(x)}{s}\vlinesub{s=0}$. We have 
$$\pl{}{s}\vlinesub{s=0}Vol(\Sigma_k(t,s)) = \int_{\Sigma_k(t)} div_{\Sigma_k(t)}(Z)\id\h^n$$
and this is a smooth function of $t$ by the definition of $\Sigma_k(t)$. We also know that this quantity is null when $t=0,1$ for all $k$, thus there exists some $t_k\in (0,1)$ such that\footnote{by following the computations in \cite[Section 9]{S83} except we have considered the two parameter variation of $\Om$. Once again we have considered $\Om\In N\cemb \R^N$ and restricted to variations which are tangent to $N$}
\begin{eqnarray}
0 &=& \pl{}{t}\vlinesub{t=t_k}\pl{}{s}\vlinesub{s=0}Vol(\Sigma_k(t,s))\nonumber\\
&=&\int_{\Sigma_k(t_k)}div_{\Sigma_k(t_k)}(V_k) +(div_{\Sigma_k(t_k)}(X_k))(div_{\Sigma_k(t_k)}(Z)) +\nonumber\\
&&+ \sum_{i=1}^n  (\D_{\tau_{k,i}} X_k)^{\bot} \cdot(\D_{\tau_{k,i}} Z)^{\bot} -\sum_{i,j=1}^n (\tau_{k,i}\cdot \D_{\tau_{k,j}} X_k)(\tau_{k,j}\cdot \D_{\tau_{k,i}} Z) \id \h^n \label{2ndv}
\end{eqnarray}
where $$X_k(\Psi^k_{t,0}(x))=\pl{\Psi^k_{t,s}(x)}{t}\vlinesub{s=0},\,\,\,\,Z(\Psi^k_{t,0}(x))=\pl{\Psi^k_{t,s}(x)}{s}\vlinesub{s=0}\,\,\,\,\text{and}\,\,\,\, V_k(\Psi^k_{t,0}(x))=\ppl{\Psi^k_{t,s}(x)}{t}{s}\vlinesub{s=0}$$
and $\{\tau_{k,i}\}$ is some orthonormal basis for $\Sigma_k(t)$ at any given point. We could have also written 
$$X_k(y)=\ed_{\nu(x)v_k(x,t)} Exp_x (\nu(x)(u^L_k(x)-u^1_k(x)))\,\,\,\,\,\,\text{where}\,\,\,\,\,\,\, y=Exp_x(\nu(x)v_k(x,t_k))$$ and $V_k=\D_{X_k} Z$.  

We also know that $v_k(x,t)\to 0$ uniformly and smoothly in $x$ as $k\to \infty$ (since $u^l_k \to 0$ uniformly and smoothly see also \cite{CM98}, \cite{S87}), for any $t\in [0,1], l\in [1,L]$ and thus $\Sigma_{k}(t)$ converges smoothly and graphically to $\Om$. Moreover $Z, V_k$ are compactly supported on each $\Sigma_k(t)$. Thus, restricting to vector fields $Z$ that are normal along $\Om$ (i.e. are written as $\eta(x)\nu(x)$ for some $\eta\in C_c^{\infty}(\Om)$), standard integration by parts and submanifold formulae lead to: that $\ti{h}_k(x)=u^L_k(x) - u^1_{k}(x)$ solves, for all $\eta\in C_c^{\infty}(\Om)$,
\begin{equation}\label{approxj}
0=\int_{\Om} \D \ti{h}_k \cdot \D \eta - (|A|^2 + Ric(\nu,\nu) )\ti{h}_k \eta  + \left(div_{\Om} (a_k \D \ti{h}_k) + b_k \cdot \D \ti{h}_k  + c_k \ti{h}_k\right)\eta
\id V_{\Om}\end{equation}
where $a_k, b_k ,c_k$ go to zero smoothly and uniformly on $\Om$ - see \cite[p. 333]{S87}\footnote{to see this notice that if \eqref{2ndv} holds on $\Om$ then we recover \eqref{approxj} with $a_k, b_k, c_k$ all zero - but as $k\to \infty$ we are converging smoothly to this situation where $\Sigma_k(t_k)$ converges to $\Om$ smoothly and graphically}. Now consider $\Om^\prime \cemb \Om$ and re-normalise $h_k(x):= \ti{h}_k(y_0)^{-1} \ti{h}_k(x)$ for some fixed $y_0\in \Om^\prime$, a Harnack estimate\footnote{since $\ti{h}_k$ is a positive solution to a sequence of uniformly elliptic equations with smooth coefficients, once the coefficients $a_k,b_k,c_k$ are sufficiently small - \cite[Corollary 8.21]{gt} gives an $L^{\infty}$ estimate and then \cite[Theorem 9.11]{gt} gives a $C^{1,\al}$ estimate (by Sobolev embedding), but then a simple bootstrapping argument gives smoothness} gives smooth control over compact subsets $\Om^{\prime\prime} \cemb\Om^\prime$. Thus we can conclude that (since $h_k>0$) we converge locally and smoothly to a non-trivial solution $h :M\sm\{y_i\}\to 
\R_{\geq 0}$ of 
\begin{equation}\label{exact}
-\Dl_{M} h - (|A|^2 + Ric(\nu,\nu))h = 0
\end{equation}
and for any $\Om\cemb M\sm\mathcal{Y}$ we have smooth estimates for $h$. Moreover the maximum principle tells us that $h>0$ on $M\sm\mathcal{Y}$.

We point out here that when the number of sheets is one then we follow the above ideas, except that we set $v_k(x,t) = t u_k(x)$ which is defined over all of $M$. Thus $\ti{h}_k(x) = u_k(x)$ solves \eqref{approxj} over all of $M$ and we can re-normalise $h_k(x):= \|\ti{h}_k\|_{L^2(M)}^{-1}\ti{h}_k(x)$. Then standard elliptic estimates\footnote{similarly to the above \cite[Theorem 9.11]{gt} gives a $W^{2,2}$ estimate, but then a simple bootstrapping argument gives smoothness} give smooth control and convergence of $h_k$ to a non-trivial limit $h$ satisfying \eqref{exact} on $M$. In the case that we have one-sided convergence then $h_k >0$ which is preserved in the limit (by the maximum principle and that $\|h_k\|_{L^2} =1$), thus we have a signed solution to the Jacobi operator $h$ and the limit must be stable (by minimax methods). In the case that we have two-sided convergence then $h$ cannot be either strictly positive or negative, therefore there must exist at least one negative eigenvalue by minimax methods.  

\paragraph{Claim 6: The solution $h$ to \eqref{exact} extends to a smooth positive solution to the Jacobi equation over all of $M$.}

To see this it suffices to check that $h$ is bounded over the $y_i\in \mathcal{Y}$ and thus\footnote{it is therefore a weak, signed solution over the whole of $M$ with a global bound, yielding smooth estimates. The maximum principle tells us that it must be strictly positive over all of $M$.} that we have a strictly positive solution to the Jacobi equation over all of $M$. In other words, $M$ must be stable by a standard minimax argument.  

We use an argument which can also be found in \cite{CM98}:  let $y_i\in \mathcal{Y}$ and for $\eta$ sufficiently small consider $\{z^1,\dots , z^n\}$ geodesic normal coordinates on $B^M_{2\eta}(y_i)$. Extend these to exponential normal coordinates $\{z^1,\dots, z^{n+1}\}$ on some small neighbourhood $U$ in $N$ about $y_i$. In $z$ coordinates consider a cylindrical neighbourhood $C_0=B_{\eta}^n\times (-c_0,c_0)$ in $\R^{n+1}$ for $c_0$ small. Given $\eps>0$ we can find $k$ sufficiently large so that the graphs $\|u^i_k\|_{C^{2,\al}(\de B^n_{\eta})} <\eps$, moreover letting $x=(z^1,\dots,z^n)$ we can extend each $u^i_k$ to the interior of $B^n_{\eta}$ by $w^i_k(x) = |x|^2u^i_k(\fr{x}{|x|})$. We can thus ensure that $\|w^i_k\|_{C^{2,\al}(B^n_{\eta})}\leq K_0\|u^i_k\|_{C^{2,\al}(\de B^n_{\eta})} <K_0\eps$. 
By the proposition in the appendix of \cite{W87}, setting $\eta$, $c_0$ and then $\eps$ sufficiently small (i.e. $k$ sufficiently large) we can foliate $C_0$ by minimal graphs (with respect to $N$) $v^i_{k,t}$ over $B_{\eta}^n$ such that 
$$\text{$v^i_{k,t}(x) = t + w^i_k(x) = t + u^i_k(x)$ for $x\in\de B_{\eta}^n$, $t\in [-c_0,c_0]$}$$ and 
$$\|v^i_{k,0}\|_{C^{2,\al}(B^n_\eta)} \leq K_1 \|u^i_k\|_{C^{2,\al}(\de B^n_\eta)}$$where the latter follows easily from the proof in \cite{W87} and the fact that $\|w^i_k\|_{C^{2,\al}(B^n_{\eta})}\leq K_0\|u^i_k\|_{C^{2,\al}(\de B^n_{\eta})}$ ($K_1$ is some uniform constant independent of $k$). We note that the result of White is only proved when $n=2$ but it follows trivially in higher dimensions - see \cite[Remark 2]{W87}. 

Setting $v^i_k = v^i_{k,0}$ we know that the difference $V_k:=v^L_k - v^1_k$ solves some uniformly elliptic differential equation (see \cite[pp. 237--238]{CM11} in the case that $n=2$)
$$L_k V_k = a^k_{ij} (V_k)_{x^ix^j} + b^k_i (V_k)_{x^i}= c_k V_k.$$ Letting $g_{ij}$ be the metric components of $N$ in $z$-coordinates, we have that the coefficients $a^k_{ij}, b^k_i, c_k$ are all uniformly controlled by $\{\|u^i_k\|_{C^{2,\al}(\de B^n_{\eta})}\}_{i=1,L}, g_{ij}, \pl{g_{ij}}{z^l}, \ppl{g_{ij}}{z^l}{z^q}$. Moreover when $\{\|u^i_k\|_{C^{2,\al}(\de B^n_{\eta})}\}_{i=1,L}$ are sufficiently small (i.e. $k$ sufficiently large), then $a^k_{ij}$ is uniformly elliptic, independently of $k$. Thus there is some $K_2>0$ such that $|a^k_{ij}|+|b_i^k| + |c_k|\leq K_2$ and $a^k_{ij}\xi^i\xi^j \geq \fr{1}{K_2}|\xi|^2$ for all $\xi\in \R^n$. By the weak maximum principle, in particular \cite[Theorem 3.7]{gt}
\begin{equation*}
\sup_{B_{\eta}^n} V_k  \leq \sup_{\de B_{\eta}^n} V_k + K_3\sup_{B_{\eta}^n} V_k\end{equation*}
where for $\eta$ sufficiently small, we can assume that $K_2\leq \fr12$. 
Therefore, when $\eta$ is sufficiently small
$$\sup_{B_{\eta}^n} V_k\leq 2 \sup_{\de B_{\eta}^n}  V_k = 2\sup _{\de B_{\eta}^n}  (u^L_k - u^1_k).$$
For $k$ sufficiently large let $M^L_k$ be the connected component of $M\cap C_0$ such that $u^L_k(x)\in M_k^L$ for some $x\in \de B^n_{\eta}$. Similarly, let $M^1_k$ be the connected component in $M\cap C_0$ corresponding to $u^1_k$.

We note by the maximum principle that $v^L_{k,t}(B^n_\eta)\cap M^L_k = \emptyset$ when $t>0$ and $v^1_{k,t}(B^n_\eta)\cap M^1_k =\emptyset$ when $t<0$. Therefore letting $D_k\In B^n_{\eta}$ denote the domain of definition of $\ti{h}_k = u^L_k - u^1_k$ we can conclude that  
$$\sup_{D_k} \ti{h}_k(x) \leq \sup_{B_\eta ^n} (v^L_k -  v^1_k) \leq 2\sup_{\de B_{\eta}^n} (u^L_k(x)- u^1_k(x))= 2\sup_{\de B_{\eta}^n} \ti{h}_k(x).$$
Thus we have that $h$ is bounded over $y_i$ and we have proved claim 6.

\end{proof}

\appendix

\section{The results of Schoen-Simon}

In \cite{ScS81} they consider functionals $F$ on balls in $\R^{n+1}$ which correspond to the volume functional of hypersurfaces in small normal coordinate balls about points in $N$. In other words, when considering minimal hypersurfaces $M\in IV_n(N)$, if we take a normal coordinate ball about some point $B^N_{\rho_0}(p)\In N$ and let $\ti{M}\In B^{n+1}_{\rho_0}(0)$ be $M\cap B^N_{\rho_0}(p)$ then they consider the functional $F$ such that 
$$Vol(M\cap B^N_{\rho_0}(p) ) = \int_{\ti{M}} F(y,\nu(y)) \id \h^n(y)$$where $\nu$ is the unit normal on $\ti{M}$ in $\R^n$. We will obviously not distinguish between $M\cap B^N_{\rho_0}(p)$ and $\ti{M}$ below, moreover we remark that if $M\cap B^N_{\rho_0}(p)$ is stable then $\ti{M}$ is stable with respect to the functional $F$ where we test against variations that are compactly supported in $B^N_{\rho_0}(p)$. See \cite{ScS81} for further details, in particular for a precise definition of the functionals $F$ under consideration there. We briefly list the necessary properties of $F$ following \cite[p. 743]{ScS81}; we assume it to be a $C^3$ function $B^{n+1}_{\rho_0}(0)\times (\R^{n+1}\sm\{0\})\to \R$  
\begin{equation}\label{1}
F(y,\lambda Z) = \lambda F(y,Z) \,\,\,\,\,\text{for all $\lambda>0$ and $(y,Z)\in B^{n+1}_{\fr{\rho_0}{2}}(0)\times (\R^{n+1}\sm\{0\})$,}
\end{equation}
and there exist $\mu$, $\mu_1$ such that
\begin{equation}\label{2}
\mu^{-1}\leq F(y,\nu)\leq \mu, \,\,\,\,\,\,\,|\D_2^\al F(y,\nu)|\leq \mu \,\,\,\,\,\,\text{for $(y,\nu)\in B^{n+1}_{\fr{\rho_0}{2}}(0)\times S^n$, $|\al|\leq 3$,}
\end{equation}
\begin{equation}\label{3}
|\D_1^\al\D_2^\beta F(y,\nu)|\leq \mu_1^{|\al|} \,\,\,\,\,\,\text{for $(y,\nu)\in B^{n+1}_{\fr{\rho_0}{2}}(0)\times S^n$, $|\al|+|\beta|\leq 3$, $0<|\al|\leq 2$}
\end{equation}
where $\al,\beta$ are multi-indices and $\D_1, \D_2$ denote differentiation with respect to the first or second component of $F$. Furthermore $F$ is the area integrand when $y=0$, i.e. 
\begin{equation}\label{4}
F(0,Z)=|Z|.
\end{equation}
Finally, for all $w\in B^{n+1}_{\fr{\rho_0}{2}}(0)$ there is a $C^3$ diffeomorphism $\psi_w:B^{n+1}_{\rho_0}(0)\to B^{n+1}_{\rho_0}(0)$ with $\psi_w(0)=w$ and 
\begin{eqnarray}\label{5}
&\sup_{B^{n+1}_{\rho_0}(0)} |\D \psi_w | +|\D \psi_w^{-1}| \leq \mu,\,\,\,\,\,\,\sup_{y\in B^{n+1}_{\rho_0}(0)} |\D^\al \psi_w |\leq \mu_1^{\al}\,\,\,\,|\al|=2,3,& \nonumber\\
&\psi_{w}^{\sharp}F(y,Z),\,\,\,\,\,\,\text{satisfies \eqref{1}-\eqref{4} above with the same $\mu, \mu_1$},&
\end{eqnarray}
where $\psi_{w}^{\sharp}F(y,Z)=F(\psi_w(y),((\ed_y \psi_w)^{-1})^{\ast}(Z))$,  $\ed_y \psi_w$ is the derivative at $y$ and $\ast$ denotes the adjoint.

\begin{theorem}\cite[Theorem 1]{ScS81}\label{ssgraph}
Suppose $F$ is a functional on $B^{n+1}_{\rho_0}(0)$ satisfying \eqref{1}-\eqref{5}, and $M$ be an orientable $C^2$ embedded hypersurface in $B^{n+1}_{\rho_0}(0)$ which is $F$-stable and satisfy $\h^n(M)\leq \mu \rho_0^n$ and $\h^{n-2}(sing(M))=0$.   
There exists a number $\dl_0 = \dl_0(n, \mu, \mu_1, \mu_1\rho_0)\in (0,1)$ such that if $x=(\ti{x}\in \R^n, x_{n+1})\in spt(M)\cap B^{n+1}_{\rho}$, $\rho\in (0,\fr{\rho_0}{4})$, $M^\prime$ is the connected component of $M$ in $C(\ti{x},\rho) = B^n_{\rho}(\ti{x})\times \R$ and
$$\sup_{y=(\ti{y}\in\R^n,y_{n+1})\in M^{\prime}} |y_{n+1} - x_{n+1}| \leq \dl_0\rho, \,\,\,\,\,\,\mu_1\rho\leq \dl_0,$$
then $M^\prime \cap C(\ti{x},\fr{\rho}{2})$ consists of a disjoint union of graphs of functions $u_1<u_2<\dots<u_L$ defined on $B^n_{\fr{\rho}{2}}(\ti{x})$ satisfying 
$$\max_i \sup_{B^n_{\fr{\rho}{2}}(\ti{x})}(|\D u_i| + \rho |\D^2 u_i|)\leq C\dl_0$$ where $C=C(n,\mu,\mu_1\rho)<\infty$.
\end{theorem}
\begin{remark}
This is clearly false if we remove the stability assumption by considering a blown down catenoid in a Euclidean ball and $F$ is exactly the area integrand of $\R^{n+1}$. The catenoid has index one and if we scale it down it eventually has index one in any ball about the origin, moreover it converges to a plane of multiplicity two - in particular we can make it satisfy all the the conditions above, except the stability condition - but it can never be a multi-valued graph at the origin. 

We also remark that the condition $\h^0(sing(M)) = 0$ when $n=2$ is sufficient for the above theorem to hold - we check this below and note that this is a trivial extension of the above result. As mentioned previously a result of Neshan Wickramasekera \cite{W14} would in particular allow one to assume only that $\h^{n-1}(sing(M))=0$ for all dimensions. 
\end{remark}

The following regularity (and compactness) theorem then holds: 
\begin{theorem}{\cite[Theorem 3]{ScS81}}\label{sschange}
Suppose $M$ is an orientable $C^2$ embedded hypersurfaces in $B^{n+1}_{\rho_0}(0)$ which is $F$-stationary and stable ($F$ as above). Moreover $\h^n(M)\leq \mu \rho_0^n$ and $\h^{n-2}(sing(M))=0$. Then $$\text{$\h^{\al}(sing(M)\cap B^{n+1}_{\fr{\rho_0}{2}}) = 0$ for all $\al >n-7$.} $$
Moreover when $n\leq 6$ there is a constant $c_1 = c_1(n,\mu,\mu_1\rho_0)$ such that 
$$\sup_{M\cap B_{\fr{\rho_0}{2}}} |A| \leq c_1 \rho_0^{-1} $$
where $A$ is the second fundamental form of $M$ in $\R^{n+1}$. 
\end{theorem}
\begin{remark}
We can therefore conclude when $2\leq n\leq 6$ that if $F_q$ are a sequence of functionals all satisfying \eqref{1}-\eqref{5} uniformly, such that they converge in $C^3$ to some limit functional $F$; then any sequence $M_q$ of $F_q$-stable hypersurfaces, with uniformly bounded mass, converge locally in the $C^2$ topology to some $M$ which is itself an $F$-stable surface. Moreover if $F_q$ converges to $F$ in $C^k$ then $M_q$ converges to $M$ in $C^{k-1, \al}$. 

Once again the assumption that $\h^0(sing(M))<\infty$ is sufficient to conclude the above. 
\end{remark}

The following Lemma is essentially proven in \cite[Lemma 1]{ScS81} - we note that this is the only place we need to check for the theorems of Schoen-Simon to hold for surfaces with point singularities.
\begin{lemma}\label{lemma}
Let $M\in IV_2(B^{3}_{\rho_0}(0))$ be $F$-stable with $\h^{0}(sing(M))<\infty$. There exist $\eps_0=\eps_0(\mu,\mu_1\rho_0)>0$ and $C=C(\mu,\mu_1\rho_0)<\infty$, such that whenever $\mu_1\rho\leq \eps_0$ and $\phi$ is a bounded locally Lipschitz function vanishing in a neighbourhood of $M\cap \de C(0,\rho)$, we have 
$$\int_M |A|^2\phi^2 \id \h^2 \leq C\left(\int_M (1-[\nu\cdot \nu_0]^2)|\D \phi|^2 \id \h^2 + \mu_1^2 \int_M \phi^2 \id \h^2 \right)$$  
\end{lemma}
for any $\nu_0\in S^2$. 
\begin{proof}

This follows by using a log cut-off argument. Notice that for any $\delta>0$ sufficiently small we can cover $sing(M)$ by $K=\h^0(sing(M))$ disjoint ambient balls $B_{r_i}(x_i)$ such that $\max_i r_i \leq \delta$ and also we will require that $B_{\sqrt{r_i}}(x_i)$ are disjoint. 

Now, it is proved in \cite{ScS81} that this Lemma is true for bounded locally Lipschitz $\phi$ supported on $reg(M)$.  

Now take any such $\phi$, not necessarily vanishing on $sing(M)$ and let 
$$\eta_{i}:=\threepartdef{0}{|x-x_i|<r_i}{\fr{\log\left(\fr{|x-x_i|}{r_i}\right)}{\log(\fr{1}{\sqrt{r_i}})}}{r_i\leq|x-x_i|<\sqrt{r_i}}{1}{|x-x_i|\geq\sqrt{r_i}}$$
where $|x-x_i|$ is the distance function for $N$ from $x_i$ - we can choose normal coordinates centred at $x_i$ in order to do this, and for $\dl$ sufficiently small. 

Now define $\psi = \phi \prod_{i=1}^K \eta_i$ which is admissible in the above, thus 
\begin{eqnarray*}
\int_M |A|^2\phi^2 \prod_{i=1}^K \eta_i^2 \id \h^2 &\leq& C\left(\int_M (1-[\nu\cdot \nu_0]^2)|\D \psi|^2 \id \h^2 + \mu_1^2 \int_M \phi^2\prod_{i=1}^K \eta_i^2 \id \h^2 \right)\\
 &\leq&  C\int_M \left((1-[\nu\cdot \nu_0]^2)(1+\mu)|\D \phi|^2 + \mu_1^2 \phi^2\right)\prod_{i=1}^K \eta_i^2\id \h^2  +\\
 &+& C(\mu) K \sup|\phi|^2  \sup_i \int_M |\D \eta_i|^2\id\h^2. 
\end{eqnarray*}
We have used Young's inequality with a ``$\mu$" and trivial estimates. 

Using this we have
\begin{eqnarray*}
\int_M |A|^2\phi^2\id \h^2 
&\leq & \lim_{\mu \to 0} \lim_{\dl \to 0}C\int_M \left((1-[\nu\cdot \nu_0]^2)(1+\mu)|\D \phi|^2 + \mu_1^2 \phi^2\right)\prod_{i=1}^K \eta_i^2\id \h^2  +\\
 &+& \lim_{\mu \to 0} \lim_{\dl \to 0}C(\mu) K \sup|\phi|^2  \sup_i \int_M |\D \eta_i|^2 \id\h^2 \\
 &=& C\left(\int_M (1-[\nu\cdot \nu_0]^2)|\D \phi|^2 \id \h^2 + \mu_1^2 \int_M \phi^2 \id \h^2 \right)
\end{eqnarray*}
where the last line follows if and only if 
$$ \int_M |\D \eta_i|^2 \id\h^2\to 0 $$ as $\dl \to 0$. 

We check this now: recall that the monotonicity formula gives us the existence of some $C$ such that for all sufficiently small $\rho$, 
$$\int_{M\cap B_\rho} \id \h^2\leq C(M)\rho^2.$$ 
Now set $\N\ni N \geq \fr{\log r_i^{-1/2}}{\log 2}$ so that $\sqrt{r_i}\leq 2^N r_i$ (we also assume $N\leq 2 \log r_i^{-1/2}$) and we have 
\begin{eqnarray*}
 \int_{M} |\D \eta_i|^2 \id \h^2 
 &\leq& \sum_{l=1}^N \int_{M\cap (B_{2^lr_i}\sm B_{2^{l-1}r_i})} \fr{1}{r^2 \log\left(\fr{1}{\sqrt{r_i}}\right)^2} \id\h^2 \\
 &\leq & \sum_{l=1}^N \fr{1}{2^{2l-2}r_i^2\log\left(\fr{1}{\sqrt{r_i}}\right)^2}\int_{M\cap B_{2^l r_i}} \id\h^2 \\
 &\leq & \sum_{l=1}^N \fr{C(M)}{\log\left(\fr{1}{\sqrt{r_i}}\right)^2} = \fr{C(M)N}{\log\left(\fr{1}{\sqrt{r_i}}\right)^2} \leq  C(M)[-\log \dl]^{-1}
\end{eqnarray*}
and the result follows. 
\end{proof}

Finally we state a theorem analogous to Theorem \ref{main} in the setting of changing ambient metrics on $N$ converging uniformly and smoothly to some limit. It can be proved using the methods in this paper, but using Theorem \ref{sschange} instead of Theorem \ref{ssreg}. Below we will fix a smooth back ground metric on $N$, call it $h$ and we say that metrics $g_k\to g$ converge smoothly to some limit if they converge smoothly with respect to $h$. Moreover, we measure volume with respect to $h$ (since this is equivalent to doing so via $g$ and thus $g_k$ also), and we let $index_k, index$ denote the index with respect to $g_k, g$ respectively.

\begin{theorem}\label{mainchange}
Let $2\leq n \leq 6$ and $N^{n+1}$ be a smooth closed manifold and $\{g_k\}_{k\in \mathbb{N}}$ a family of Riemannian metrics on $N$ converging smoothly to some limit $g$. If $\{M_k^n\}\In N$ is a sequence of closed, connected and embedded minimal hypersurfaces in $(N,g_k)$ with 
$$\h^n (M_k) \leq \Lambda <\infty \,\,\,\,\,
\text{and}\,\,\,\,\,\, index_k(M_k)\leq I$$
for some fixed constants $\Lambda\in \R$, $I\in \N$ independent of $k$. Then up to subsequence, there exists a closed connected and embedded minimal hypersurface $M\In (N,g)$ where $M_k \to M$ in the varifold sense with$$\h^n (M) \leq \Lambda <\infty \,\,\,\,\,
\text{and}\,\,\,\,\,\, index(M)\leq I.$$
We have that the convergence is smooth and graphical for all $x\in M\sm  \mathcal{Y}$ where $\mathcal{Y}=\{y_i\}_{i=1}^K\In M$ is a finite set with $K\leq I$ and the following dichotomy holds:\begin{itemize}
\item if the number of leaves in the convergence is one then $\mathcal{Y}=\emptyset$ i.e. the convergence is smooth and graphical everywhere

\item if the number of sheets is $\geq 2$ 
\begin{itemize}
\item if $N$ has $Ric_N >0$ then $M$ cannot be one-sided
\item if $M$ is two-sided then $M$ is stable. 
\end{itemize}
\end{itemize}

\end{theorem}

Clearly then we recover Corollary \ref{CS} in the case of changing background metrics: 
\begin{cor}\label{CS}
Let $N^{n+1}$ be a closed Riemannian manifold with $Ric_N >0$ and $2\leq n\leq 6$. Denote by $\mathfrak{M}_k^n(N)$ the class of closed, smooth, and embedded minimal hypersurfaces $M\In (N,g_k)$. Then given any $0<\Lambda <\infty$, $I\in \mathbb{N}$ and any sequence $M_k\in \mathfrak{M}^n_k$ with $\h^n(M_k)\leq \Lambda$ and $index_k(M_k)\leq I$ then there exists some $M\in \mathfrak{M}^n(N)$ such that (up to subsequence) $M_k\to M$ in the $C^k$ topology for all $k\geq 2$ with single-sheeted graphical convergence. Moreover $\h^n(M)\leq \Lambda$ and $index(M)\leq I$. 
\end{cor}

\end{document}